\documentclass[10pt, a4paper]{article}
\usepackage[french, english]{babel}
\usepackage[utf8,latin1]{inputenc}
\usepackage{amsmath}
\usepackage{amssymb}
\usepackage{amsthm}
\usepackage{latexsym,tikz}
\usepackage{hyperref}
\usepackage{enumerate}
\usepackage[all]{xy}
\usepackage{mathrsfs}
\usepackage{dsfont}
\usepackage{tikz}
\usetikzlibrary{decorations.pathmorphing}
\usepackage{appendix}
\usetikzlibrary{decorations.pathreplacing}
\usepackage{pdfpages}

\usepackage[nottoc]{tocbibind}

\newtheorem{thm}{Theorem}[section]
\newtheorem{cor}[thm]{Corollary}
\newtheorem{lem}[thm]{Lemma}
\newtheorem{prop}[thm]{Proposition}
\newtheorem{defn}[thm]{Definition}

\newtheorem{rem}[thm]{Remark}

\def\det{\mathrm{det}}

\def\cF{\mathcal{F}}
\def\rG{\mathrm{G}}

\def\bG{\textbf{G}}

\def\GL{\mathrm{GL}}

\def\Hom{\mathrm{Hom}}

\def\ind{\mathrm{ind}}

\def\cK{\mathcal{K}}

\def\rL{\mathrm{L}}
\def\bL{\textbf{L}}

\def\rM{\mathrm{M}}

\def\rN{\mathrm{N}}

\def\rO{\mathrm{O}}
\def\rP{\mathrm{P}}

\def\res{\mathrm{res}}

\def\SL{\mathrm{SL}}

\def\rT{\mathrm{T}}

\def\bT{\mathbf{T}}

\def\rU{\mathrm{U}}

\def\rV{\mathrm{V}}

\def\rY{\mathrm{Y}}
\def\cY{\mathcal{Y}}
\def\rZ{\mathrm{Z}}

\begin{document}





\hypersetup{							
pdfauthor = {Peiyi Cui},			
pdftitle = {representation mod l of SL_n},			
pdfkeywords = {Tag1, Tag2, Tag3, ...},	
}					
\title{Supercuspidal support of irreducible modulo $\ell$-representations of $\mathrm{SL_n}(F)$}
\author{Peiyi Cui \footnote{peiyi.cuimath@gmail.com, Institut de recherche math\'ematique de Rennes, Universit\'e de Rennes 1 Beaulieu, 35042 Rennes CEDEX, France; Current address: Department of Mathematics, Oskar-Morgenstern-Platz 1, 1090 Vienna, Austria.}}

\date{\vspace{-2ex}}
\maketitle

\begin{abstract}
Let $k$ be an algebraically closed field with characteristic $\ell\neq p$. We show that the supercuspidal support of irreducible smooth $k$-representations of Levi subgroups $\rM'$ of $\mathrm{SL}_n(F)$ is unique up to $\rM'$-conjugation, where $F$ is either a finite field of characteristic $p$ or a non-archimedean locally compact field of residual characteristic $p$.
\end{abstract}

\tableofcontents

\section{Introduction}
Let $F$ be a non-archimedean locally compact field with residual characteristic $p$, and $k$ be an algebraically closed field with characteristic $\ell\neq p$. Let $\bG$ be a connected reductive group defined over $F$ or a finite field $\mathbb{F}_q$, where $q$ is a power of $p$. Denote by $\rG$ the group of $F$ (resp. $\mathbb{F}_q$) rational points $\bG(F)$ (resp. $\rG(\mathbb{F}_q)$). 

The supercuspidal support of an irreducible smooth $k$-representation $\pi$ of $\rG$ is important during the study of the theory of representations of $\rG$. When $\ell$ is equal to $0$, supercuspidal representations are all cuspidal, and there is a quick proof that the cuspidal support of $\pi$ is unique up to $\rG$-conjugation. When $\ell$ is positive, an example of cuspidal but not supercuspidal $k$-representation has been found in \cite{V1} when $\bG=\mathrm{GL}_2$. The cuspidal support of an irreducible $k$-representation of $\rG$ is always unique up to $\rG$-conjugation, while the supercuspidal support is not, and an example of non-uniqueness has been found when $\bG=\mathrm{Sp}_8$ in \cite{Da}. However, the uniqueness of supercuspidal support is true when $\bG=\mathrm{GL}_n$ and a proof was given in \cite{V2}. In this article, the same result is proved when $\bG=\mathrm{SL}_n$ and its Levi subgroups. The uniqueness of supercuspidal support is the base stone of the Bernstein decomposition of the category $\mathrm{Rep}_k(\mathrm{GL}_n(F))$ of smooth $k$-representations of $\mathrm{GL}_n(F)$ in \cite{Helm}. Since we obtain the same result for $\mathrm{SL}_n(F)$, it shows a substantial possibility that the category $\mathrm{Rep}_k(\mathrm{SL}_n(F))$ can be decomposed relative to inertially equivalent classes of supercuspidal pairs of $\mathrm{SL}_n(F)$ (see Definition \ref{definition 000} for supercuspidal pair and two supercuspidal pairs are said to be inertially equivalent if they are $\mathrm{SL}_n(F)$-conjugate to each other up to an unramified character). 

This article is the first step of generalising the result in \cite{Helm} to $\mathrm{SL}_n(F)$. Let $W(k)$ be the ring of Witt vectors of $k$. In \cite{Helm}, Helm gave a Bernstein decomposition of $\mathrm{Rep}_{W(k)}(\mathrm{GL}_n(F))$, from which one can deduce the Bernstein decomposition of $\mathrm{Rep}_k(\mathrm{GL}_n(F))$. It is worth noting that firstly the coefficient $W(k)$ is essentially needed for his later work with Emerton of the local Langlands correspondence for $\mathrm{GL}_n$ in famillies, secondly his proof relies on a family of injective objects in $\mathrm{Rep}_{W(k)}(\mathrm{GL}_n(F))$, which is constructed from the $W(k)$-projective covers of cuspidal $k$-representations of $\mathrm{GL}_n(\mathbb{F}_q)$. In this article, we consider $k$-representations of Levi subgroups $\rM'$ of $\mathrm{SL}_n(\mathbb{F}_q)$ in Section \ref{chapter 3}, where we also study $W(k)[\rM']$-modules and the $W(k)$-projective covers of cuspidal $k$-representations. Since the fractional field $\mathcal{K}$ of $W(k)$ must not be sufficient large for finite group $\mathrm{GL}_n(\mathbb{F}_q)$, we need more discussion about this coefficient here. In Section \ref{section 4}, we consider $k$-representations of Levi subgroups of $p$-adic groups $\mathrm{SL}_n(F)$.

To be more precisely, in Section \ref{chapter 3}, for an irreducible cuspidal $k$-representation $\nu$ of a Levi subgroup $\rM'$, the $W(k)[\rM']$-projective cover $\mathcal{P}_{\nu}$ of $\nu$ can be constructed from Gelfand-Graev $W(k)$-lattice. A computation of $r_{\rL'}^{\rM'}\mathcal{P}_{\nu}$ gives the uniqueness of supercuspidal support of $\nu$, where $\rL'$ denotes a Levi subgroup of $\rM'$ and $r_{\rL'}^{\rM'}$ denotes the normalised parabolic restriction relative to $\rL'$.

In Section \ref{section 4}, a basic fact is that for an irreducible cuspidal $k$-representation $\pi'$ of $\rM'$, there is an irreducible cuspidal $k$-representation $\pi$ of $\rM$ such that $\pi\vert_{\rM'}$ contains $\pi'$ as a sub-representation, where $\rM$ is a Levi subgroup of $\rG=\mathrm{GL}_n(F)$. Let $(\rL,\tau)$ belongs to the supercuspidal support of $\pi$, and $\tau\vert_{\rL'=\rL\cap\rG'}\cong\oplus_{i\in I}\tau_i$ where $I$ is finite. We first prove that the supercuspidal support of $\pi'$ is contained in $\cup_{i\in I}(\rL',\tau_i)_{\rM'}$, where $(\rL',\tau_i)_{\rM'}$ denotes the $\rM'$-conjugacy class of $(\rL',\tau_i)$. Then we generalise the operator of derivative defined in \cite{BeZe} for $\mathrm{GL}_n$ to $\rM'$. Since $\rM'$ can not be written as a direct product of special linear groups in lower rank, the author can not find a way to avoid the complication of notations in Section \ref{subsubsection 5.1.1}. At the end, we deduce that there exists one unique $i_0\in I$ such that $(\rL',\tau_{i_0})$ belongs to the supercupidal support of $\pi'$ from the unicity of Whittaker model of $\pi$.

This is a part of the thesis of the author. She would like to thank Anne-Marie Aubert for her guidance and patient encouragement, and thank Vincent S\'echerre for the careful reading and helpful comments about the content as well as the writing of early version of this article.

\section{Cuspidal and supercuspidal representations}
\paragraph{Basic notations}
Let $F$ be a non-archimedean locally compact field with residual characteristic $p$, and $k$ be an algebraically closed field with characteristic $\ell\neq p$. Let $\bG$ be a connected reductive group defined over $F$ or $\mathbb{F}_q$, where $q$ is a power of $p$, and $\rG$ be the group of $F$ (resp. $\mathbb{F}_q$) rational points of $\bG$. In this article, a $k$-representation of $\rG$ is always assumed to be smooth.

Fix a Borel subgroup of $\rG$. Let $\rM$ be a standard Levi subgroup of $\rG$. Denote by $i_{\rM}^{\rG}$ and $r_{\rM}^{\rG}$ the normalised parabolic induction and normalised parabolic restriction. Let $K$ be a closed subgroup of $\rG$. Denote by $\ind_{K}^{\rG}$ the compact induction from $K$ to $\rG$, and $\res_{K}^{\rG}$ the restriction from $\rG$ to $K$.

\begin{defn}
\label{definition 000}
Let $\pi$ be an irreducible $k$-representation of $\rG$, we say 
\begin{itemize}
\item $\pi$ is cuspidal, if for any proper Levi subgroup $\rM$ and irreducible $k$-representation $\sigma$ of $\rM$, $\pi$ does not appear as a sub nor a quotient-representation of $i_{\rM}^{\rG}\sigma$;
\item $\pi$ is supercuspidal, if for any proper Levi subgroup $\rM$ and irreducible $k$-representation $\sigma$ of $\rM$, $\pi$ does not appear as a subquotient of $i_{\rM}^{\rG}\sigma$.
\end{itemize}
We say a pair $(\rM,\sigma)$ consisting with a Levi subgroup $\rM$ and an irreducible $k$-representation $\sigma$ is a cuspidal (resp. supercuspidal) pair, if $\sigma$ is cuspidal (resp. supercuspidal). We say
\begin{itemize}
\item a cuspidal pair $(\rM,\sigma)$ belongs to the cuspidal support of $\pi$, if $\pi$ is a sub or a quotient-representation of $i_{\rM}^{\rG}\pi$.
\item a supercuspidal pair $(\rM,\sigma)$ belongs to the supercuspidal support of $\pi$, if $\pi$ is a sub-quotient of $i_{\rM}^{\rG}\sigma$.
\end{itemize}
\end{defn}

\begin{rem}
In the above definition, $\pi$ being cuspidal is equivalent with $r_{\rM}^{\rG}\pi$ being zero for any proper Levi $\rM$ of $\rG$. 
\end{rem}

\paragraph{Reduction to cuspidal cases}
Let $\pi$ be an irreducible $k$-representation of $\rG$. The cuspidal support of $\pi$ is unique up to $\rG$-conjugation as proved in \cite{V1}. To prove the uniqueness of supercuspidal support of $\pi$, it is enough to prove the same result for each irreducible cuspidal $k$-representations of Levi subgroups of $\rG$. In fact, let $(\rM,\sigma)$ be a cuspidal pair inside the cuspidal support of $\pi$, and $(\rL,\tau)$ be a supercuspidal pair inside the supercuspidal support of $\pi$. Assume $\rM$ and $\rL$ are standard. We deduce that $\sigma$ is a sub-quotient of $r_{\rM}^{\rG}i_{\rL}^{\rG}\tau$, to which apply the filtration given in \cite[\S II,2.18]{V1}, then we obtain that up to a conjugation of $w$, which is an element of the Weyl group of $\rG$, $(w(\rL),w(\tau))$ belongs to the supercuspidal support of $\sigma$.

\section{$k$-representations of finite groups $\mathrm{SL}_n(F)$}
\label{chapter 3}

In this section, let $\textbf{G}'=\mathrm{SL}_n$ and $\bG=\mathrm{GL}_n$ be defined over $\mathrm{F}_q$, where $q$ is a power of a prime number $p$. Denote by $\rG'=\textbf{G}'(\mathrm{F}_q)$ and $\rG=\bG(\mathrm{F}_q)$. Recall that $k$ is an algebraically closed field with characteristic $\ell\neq p$. Let $W(k)$ be the ring of Witt vectors of $k$ and $\mathcal{K}$ the fractional field of $W(k)$, and $\overline{\mathcal{K}}$ an algebraic closure of $\mathcal{K}$. We have two main purposes in this section, one is to prove Theorem \ref{theorem 1}. The other one is to construct the $W(k)[\rM']$-projective cover of an irreducible cuspidal $k$-representation of $\rM'$, where $\rM'$ denotes a Levi subgroup of $\rG'$.

Notice that the center of $\bG'$ is disconnected but the center of $\bG$ is connected, so we follow the method of \cite{DeLu} (page $132$), which is also applied in \cite{Bonn}: consider the regular inclusion $i: \bG' \rightarrow \bG$, then we want to use functor $\mathrm{Res}_{\rG'}^{\rG}$ to deduce properties from $\rG$-representations to $\rG'$-representations. 

\subsection{Projective modules}
\paragraph{Regular inclusion $i$}\
We summarize the context we will need in Section 2 of [Bon]:

Let $\cF$ be the Frobenius morphism of the Galois group $\mathrm{Gal}(\overline{\mathbb{F}}_q\slash \mathbb{F}_q)$, where $\overline{\mathbb{F}}_q$ is an algebraic closure of $\mathbb{F}_q$. $\cF$ induces an isogeny of $\bG$, which we also denote by $\cF$. In particular, the invariant group $\bG^{\cF}=\rG$. The canonical inclusion $i$ from $\bG'$ to $\bG$ commutes with $\cF$ and maps $\cF$-stable maximal torus to $\cF$-stable maximal torus. If we fix one $\cF$-stable maximal torus $\bT$ of $\bG$ and denote by $\bT'= i^{-1}(\bT)$, then $i$ induces a bijection between the root systems of $\bG$ and $\bG'$ relative to $\bT$ and $\bT'$. Furthermore, $i$ gives a bijection between standard $\cF$-stable parabolic subgroups of $\bG$ and $\bG'$ with inverse $\cdot\cap \bG'$, which respects subsets of simple roots contained by parabolic subgroups. Besides, restricting $i$ to a $\cF$-stable Levi subgroup $\bL$ of a $\cF$-stable parabolic subgroup of $\bG$ is the canonical inclusion from $\bL'$ to $\bL$. 

From now on, we fix a $\cF$-stable maximal torus $\bT_0$ of $\bG$, and fix $\bT'_0=i(\bT_0)$ of $\bG'$ as well. For any $\cF$-stable standard Levi subgroup $\bL$, we always denote by $\bL'=i(\bL)$, and denote by $\rL$ and $\rL'$ the corresponding split Levi subgroups $\bL^{\cF}$ and $\bL'^{\cF}$ respectively.

Now we consider the dual groups. Let $(\bG^{\ast}, \bT_0^{\ast}, \cF^{\ast} )$ and  $(\bG'^{\ast}, \bT'^{\ast}, \cF^{\ast})$ be triples dual to $(\bG, \bT_0, \cF)$ and $(\bG', \bT', \cF)$, where $\bG^{\ast}$ is dual to $\bG$ and $\cF^{\ast}$ is the dual isogeny of $\cF$. We deduce a canonical surjective morphism $i^{\ast}: \bG^{\ast} \rightarrow \bG'^{\ast}$, which commutes with $\cF^{\ast}$ and maps $\bT_0^{\ast}$ to $\bT_0^{'\ast}$. For any $\cF$-stable standard parabolic subgroup $\textbf{P}$ and its $\cF$-stable Levi soubgroup $\bL$, we use $\textbf{P}'$ and $\bL'$ to denote the $F$-stable standard parabolic subgroups $\textbf{P}\cap\bG'$ and Levi subgroups $\bL\cap\bG'$, then we have:
$$i^{\ast} (\bL^{\ast}) = \bL'^{\ast}.$$
After denoting by $\bL'^{\ast^{\cF^{\ast}}}= \rL'^{\ast}$ and by $\bL^{\ast^{\cF^{\ast}}}= \rL^{\ast}$, we have:
$$i^{\ast} (\rL^{\ast}) = \rL'^{\ast}.$$

\paragraph{Lusztig series and $\ell$-blocks}\
From now on, if we consider a semisimple element $\tilde{s} \in \rL^{\ast}$ for any split Levi subgroup $\rL^{\ast}$ of $\rG^{\ast}$, we always denote by $s$ the image $i^{\ast}(\tilde{s})$ and by $[\tilde{s}]$ the $\rL^{\ast}$-conjugacy class of $s$, and a similar definition for $[s]$. We say a semisimple element is $\ell$-regular if $\ell$ does not divide its order. Since the order of $\tilde{s}$ is divisible by the order of $s$, we have that $s$ is $\ell$-regular if $\tilde{s}$ is $\ell$-regular. By the theory of Delign-Lusztig, an irreducible $k$-representation $\pi$ of $\rL$ corresponds to a semisimple conjugacy class $[\tilde{s}]$, where $\tilde{s}$ is $\ell$-regular.

Let $\textbf{G} (\mathbb{F}_q)$ be a finite group of Lie type, where $\textbf{G}$ is a connected reductive group defined over $\mathbb{F}_q$. For any irreducible representation $\chi$ of $\textbf{G} (\mathbb{F}_q)$, let $e_{\chi}$ denote the central idempotent of $\overline{\cK}(\textbf{G}(\mathbb{F}_q))$ associated to $\chi$ (see definition in the beginning of \cite[\S 2]{BrMi}). Fixing a semisimple element $s \in \textbf{G}^{\ast}(\mathbb{F}_q)$, let $\mathcal{E}(\textbf{G} (\mathbb{F}_q),(s))$ be the Lusztig serie of $\textbf{G} (\mathbb{F}_q)$ corresponding to $[s]$. If $s$ is $\ell$-regular, define
$$\mathcal{E}_\ell (\textbf{G} (\mathbb{F}_q),s) := \bigcup_{t \in (C_{\textbf{G}^{\ast}}(s)^{F^{\ast}} )_\ell} \mathcal{E}(\textbf{G}(\mathbb{F}_q), (ts)).$$
Here $(C_{\textbf{G}^{\ast}}(s)^{\cF^{\ast}} )_\ell$ denotes the group consisting with all $\ell$-elements of $C_{\textbf{G}^{\ast}}(s)^{\cF^{\ast}}$, where $C_{\textbf{G}^{\ast}}(s)$ is the centraliser group of $s$ in $\bG$, and $ts$ is semisimple as well. Now define:
$$b_s= \sum_{\chi \in \mathcal{E}_\ell(\textbf{G}(\mathbb{F}_q),s)} e_{\chi},$$ 
which obviously belongs to $\overline{\cK}(\textbf{G}(\mathbb{F}_q))$. 

For the convenience reason, we state a theorem in \cite{BrMi} below.
\begin{thm}[Brou$\mathrm{\acute{e}}$, Michel]
\label{theorem Mi,Br}
Let $s \in \rG^{\ast}$ be a semisimple $\ell$-regular element, and $\mathcal{L}'$ be the set of prime numbers except $\ell$. Define $\overline{\mathbb{Z}}_\ell = \overline{\mathbb{Z}}[1/r]_{r\in \mathcal{L}'}$, where $\overline{\mathbb{Z}}$ denotes the ring of algebraic integers, then $b_s \in \overline{\mathbb{Z}}_\ell [\rG]$.  
\end{thm}


\begin{rem}
\label{remark 0009}
We view $\overline{\mathbb{Z}}_{\ell}$ as a subring of $\overline{\mathcal{K}}$. Let $K^{unr}$ be an unramified closure in $\overline{\mathcal{K}}$ of $\mathbb{Q}_{\ell}$ and $O^{unr}$ be the ring of integers of $K^{unr}$, in fact we have $b_s\in O^{unr}[\rG]$. By the proof of Theorem 9.12 of \cite{CE}, we know that the support of $b_s$ is contained in $\rG_{\ell'}$, which is the set of elements in $\rG$ whose order is prime to $\ell$. Hence we have that $b_s\in K^{unr}[\rG]\cap\overline{\mathbb{Z}}_{\ell}[\rG]=O^{unr}[\rG]$. In particular,  $O^{unr}\subset W(\overline{\mathbb{F}}_{\ell})\subset W(k)$, since $W(\overline{\mathbb{F}}_{\ell})$ is the completion of the $\ell$-adic topology of an unramified closure of $\mathbb{Q}_{\ell}$. It is worth noting that $K^{unr}$ and $\mathcal{K}$ must not be sufficient large for $\rG$. In particular, there exists cuspidal $\overline{\mathcal{K}}$-representation of $\mathrm{GL}_2(\mathbb{F}_q)$ which is not defined over $\mathcal{K}$.
\end{rem}


\begin{prop}
\label{proposition 2.3}
For any split Levi subgroup $\rL$ (resp. $\rL'$) and any semisimple $\ell$-regular element $\tilde{s} \in \rL^{\ast}$ (resp. $s \in \rL'^{\ast}$), we have: $b_{\tilde{s}} \in O^{unr}[\rL]$ (resp. $b_s \in O^{unr}[\rL']$).
\end{prop}

\begin{proof}
We deduce from the analysis above and the definition that $e_{\chi} \in K[\rL]$. Combining this with Theorem \ref{theorem Mi,Br}, we conclude that $b_{\tilde{s}} \in O^{unr}[\rL]$. The same for $b_s$. 
\end{proof}

\paragraph{Gelfand-Graev lattices and its projective direct summands}\
In this section, we construct the projective cover of an irreducible cuspidal $k$-representation of $\rL'$ by using Gelfand-Graev lattice, and prove that it is a direct summand of the projective cover of an irreducible cuspidal $k$-representaions of $\rL$ after restricted to $\rL'$ (see Proposition \ref{lemma a.1}). 

For a split Levi subgroup $\rL'$ of $\rG'$, fix a rational split Borel subgroup $\mathrm{B}_{\rL'}'$ with unipotent radical $\rU_{\rL'}$. Denote by $\mathrm{O_{U}}(\rL')$ the set of non-degenerate $\overline{\mathcal{K}}$-characters of $\rU_{\rL'}$. Let $\rL$ be the Levi subgroup of $\rG$ such that $\rL\cap\rG'=\rL'$, the group $\rU_{\rL'}=\rU_{\rL}$ is also the unipotent radical of $\rL$. Notice that $\mathrm{O_{U}}(\rL)=\mathrm{O_{U}}(\rL')$ consists only one unique $\rL$-conjugacy class, but multiple $\rL'$-conjugacy classes. 

Let $(K,O,k)$ be a splitting $\ell$-modular system of $\rG$, where $K$ is a finite extension of $\mathcal{K}$ sufficiently large for $\rG$ and $O$ is its ring of integers. For an $\mu \in \rO_{\rU} (\rL')$, it contains an $O [\rU_{\rL'}]$-lattice, and we denote it by $\rO_{\mu}$. Define $\rY_{\rL', \mu} = \ind_{\rU_{\rL'}}^{\rL'} \rO_{\mu}$, the \textbf{Gelfand-Graev lattice} associated to $\mu$. In fact, we have that $\rY_{\rL', \mu}$ is defined up to the $\rT'$-conjugacy class of $\mu$. Take any $\ell$-regular semisimple element $s \in \rL'^{\ast}$, define:
$$\rY_{\rL', \mu, s} = b_s \cdot \rY_{\rL', \mu}.$$
Meanwhile, from the definition we have directly that 
$$\sum_{[s]}b_s=1,$$
where the sum runs over all the $\ell$-regular semisimple $\rL'^{\ast}$-conjugacy class $[s]$. So:
$$\rY_{\rL', \mu} = \sum_{[s]} \rY_{\rL', \mu, s}.$$
Since $\rO_{\mu}$ is projective and the compact induction respects projectivity, we know that $\rY_{\rL', \mu}$ is a projective $O[\rL']$-module. Proposition \ref{proposition 2.3} implies that $\rY_{\rL', \mu, s}$ are  $O[\rL']$-modules, which are direct components of projective $O[\rL']$-module $\rY_{\rL', \mu}$, hence $\rY_{\rL', \mu, s}$ are also projective $O[\rL']$-modules. 

We define 
$$\mathcal{E}_{\ell'}(\rG):=\bigcup_{z \text{ semi-simple, }\ell-\text{regular}}\mathcal{E}(\rG,z)$$

\begin{defn}[Gruber, Hiss]
\label{defn fin aa}
Let $\rG$ be the group of $\mathbb{F}_q$-points of an algebraic group defined over $\mathbb{F}_q$, and $(K,O,k)$ be a splitting $\ell$-modular system. Let $Y$ be an $O[\rG]$-lattice with ordinary character $\psi$. Write $\psi=\psi_{\ell'}+\psi_{\ell}$, such that all constituents of $\psi_{\ell'}$ and non of $\psi_{\ell}$ belong to $\mathcal{E}_{\ell'}(\rG)$. Then there exists a unique pure sublattice $V\leq Y$, such that $Y\slash V$ is an $O[\rG]$-lattice whose character is equal to $\psi_{\ell'}$. The quotient $Y\slash V$ is called the $\ell$-regular quotient of $Y$ and denoted by $\pi_{\ell'}(Y)$.
\end{defn}

\begin{cor}
\label{corollary 3.4}
Let $\rL'$ be a split Levi subgroup of $\rG'$, and $s$ be an $\ell$-regular semisimple element in $\rL'^{\ast}$. For any $\mu \in \rO_{\mathrm{U}}(\rL')$, the module $\rY_{\rL', \mu, s}$ is indecomposable.
\end{cor}

\begin{proof}
Since $\rY_{\rL', \mu, s}$ is a projective $O [\rL']$-module, the section \S $4.1$ of \cite{GrHi} or Lemma 5.11(Hiss) in \cite{Geck} tells us that it is indecomposable if and only if its $\ell$-regular quotient $\pi_{l'} (\rY_{\rL', \mu, s})$ (see \S 3.3 in \cite{GrHi}) is indecomposable. Inspired by section 5.13. of \cite{Geck}, we consider $K \otimes \pi_{l'}(\rY_{\rL', \mu, s})$, which is the unique irreducible sub-representation of $K \otimes \rY_{\rL', \mu}$ lying in Lusztig serie $\mathcal{E}(\rL', (s))$. The module $\pi_{\ell'}(\rY_{\rL',\mu,s})$ is torsion-free, so we deduce that $\pi_{l'}(\rY_{\rL', \mu, s})$ is indecomposable.
\end{proof}

\begin{prop}
\label{proposition 2.5}
Let $\rL'$ be a split Levi subgroup of $\rG'$, and $\mu \in \rO_{\rU}(\rL')$. All the projective indecomposable direct summands $\rY_{\rL', \mu ,s }$ of Gelfand-Graev lattice $\rY_{\rL',\mu}$ are defined over $O^{unr}$. In particular, there exist indecomposable projective $W(k)[\rL']$-modules $\mathcal{Y}_{\rL',\mu,s}$ such that $\mathcal{Y}_{\rL',\mu,s} \otimes_{W(k)} O\cong \rY_{\rL', \mu, s}$.
\end{prop}

\begin{proof}
Since $\rU_{\rL'}$ are $p$-groups, $\mu$ is defined over $O^{unr}$, which is equivalent to say that there is a $O^{unr}[\rU_{\rL'}']$-module $\mathcal{O}_{\mu}$ such that $\rO_{\mu}= \mathcal{O}_{\mu} \otimes_{O^{unr}} O$. Define a projective $O^{unr}[\rL']$-module $\mathrm{Ind}_{\rU_{\rL'}'}^{\rL'}(\mathcal{O}_{\mu})$. Denote by $\mathcal{Y}_{\rL',\mu}= \mathrm{Ind}_{\rU_{\rL'}'}^{\rL'}(\mathcal{O}_{\mu})\otimes_{O^{unr}} W(k)$. Since $k$ is algebraically closed, $\overline{\rY}_{\rL',\mu}=\rY_{\rL',\mu}\otimes_{\rO}k$ coincides with $\overline{\mathcal{Y}}_{\rL',\mu}=\mathcal{Y}_{\rL',\mu}\otimes_{W(k)}k$. By Remark \ref{remark 0009}, we define $\mathcal{Y}_{\rL',\mu,s}=b_s\mathcal{Y}_{\rL',\mu}$, which is indecomposable from the fact that $\mathcal{Y}_{\rL',\mu,s}\otimes_{W(k)}k=\rY_{\rL',\mu,s}\otimes_{O}k$ is indecomposable.
\end{proof}

\begin{rem}
Let $\mathrm{B}_{\rL}$ be a split Borel subgroup of $\rL$, such that $\mathrm{B}_{\rL}\cap\rL'=\mathrm{B}_{\rL'}'$. Since $\rU_{\rL'}'$ is also the unipotent radical of $\mathrm{B}_{\rL}$. We can repeat the proof for $\rY_{\rL,\tilde{s}}$ and see that they are also defined over $O^{unr}$.
\end{rem}

After the above discussion, we consider the $\ell$-modular system $(\cK,W(k),k)$ instead of a splitting system $(K,O,k)$. For a split Levi subgroup $\rL$ of $\rG$, since the set $\mathrm{O_{U}}(\rL)$ consists with only one orbit under conjugation of a split maximal torus of $\rL$, the Gelfand-Graev $W(k)$-lattice is unique, and we denote it by $\cY_{\rL}$. All the discussion above work for $\cY_{\rL}$. In particular, for an $\ell$-regular semisimple element $\tilde{s} \in \rL^{\ast}$, we denote by $\cY_{\rL, \tilde{s}}$ the indecomposable projective direct summand $b_{\tilde{s}} \cdot \cY_{\rL}$. Now we study the relation between $\cY_{\rL, \tilde{s}}$ and $\cY_{\rL', \mu, s}$.


\begin{cor}
\label{corollary 3.1}
Let $\tilde{s} \in \rL^{\ast}$ be a semisimple $\ell$-regular element, then:
$$\mathrm{res}_{\rL'}^{\rL} (b_{\tilde{s}} \cdot \cY_{\rL}) \hookrightarrow b_{s} \cdot \mathrm{res}_{\rL'}^{\rL} \cY_{\rL}.$$
\end{cor}

\begin{proof}
We know directly from definition that for any semisimple $\ell$-regular $s' \in \rG'^{\ast}$:
$$b_{s'} \cdot \mathrm{res}_{\rL'}^{\rL} (b_{\tilde{s}} \cdot \cY_{\rL}) \hookrightarrow b_{s'} \cdot \mathrm{res}_{\rL'}^{\rL} \cY_{\rL},$$
Meanwhile $b_{s'} \cdot \mathrm{res}_{\rL'}^{\rL} (b_{\tilde{s}} \cdot \cY_{\rL})$ is a projective $W(k) [\rG']$-module, so it is free over $W(k)$. Proposition 11.7 in \cite{Bonn} told us that $b_{s'} \cdot \mathrm{res}_{\rL'}^{\rL} (b_{\tilde{s}} \cdot \cY_{\rL}) \otimes \overline{\cK} = 0$ if $[s'] \neq [s]$ with $s= i^{\ast}(\tilde{s})$, which means $b_{s'} \cdot \mathrm{res}_{\rL'}^{\rL} (b_{\tilde{s}} \cdot \cY_{\rL})  = 0$.
Combine this with
$$\bigoplus_{[s']} b_{s'} \cdot \mathrm{res}_{\rL'}^{\rL} (b_{\tilde{s}} \cdot \cY_{\rL})= \mathrm{res}_{\rL'}^{\rL} (b_{\tilde{s}} \cdot \cY_{\rL}),$$
where $[s']$ run over the semisimple conjugacy classes of $\rL^{\ast}$. We obtain the result.
\end{proof}

\begin{prop}
\label{proposition 3}
For a split Levi subgroup $\rL$ of $\rG$, let $\rL'$ be the split Levi subgroup $\rL\cap \rG'$ of $\rG'$. Denote by $\rZ(\rL)$ and $\rZ(\rL')$ the center of $\rL$ and $\rL'$ respectively. We have an equation:
$$\mathrm{res}_{\rL'}^{\rL} \cY_{\rL} = \vert \rZ(\rL) : \rZ(\rL') \vert \bigoplus_{[\mu] \in \rO_{\mathrm{U}}(\rL')} \cY_{\rL',\mu},$$
where $[\mu]$ denote the $\rT'$-orbit of $\mu$.
\end{prop}

\begin{proof}
Let $\mathrm{B}$ be a split Borel subgroup of $\rL$ and $\mathrm{B}'= \mathrm{B} \cap \rL'$ the corresponding split Borel of $\rL'$, and $\rU'$ denotes the unipotent radical of $\mathrm{B}'$, observing that $\rU'$ is also the unipotent radical of $\mathrm{B}$. Fixing one non-degenerate character $\mu$ of $\rU'$, let $\rO_{\mu}$ be its $W(k) [\rU']$-lattice. By the transitivity of induction, we have:
$$\cY_{\rL}= \mathrm{ind}_{\rL'}^{\rL} \circ \mathrm{ind}_{\rU'}^{\rL'} \rO_{\mu} = \mathrm{ind}_{\rL'}^{\rL} \cY_{\rL', \mu}.$$
Since $[\rT:\rT']=[\rL: \rL']$, by using Mackey formula we have:
$$\mathrm{res}_{\rL'}^{\rL} \cY_{\rL} = \bigoplus_{\alpha_i \in [\rT : \rT']} ad(\alpha_i) ( \cY_{\rL', \mu}),$$
where $ad(\cdot)$ denotes the conjugation operator. Furthermore, $ad(\alpha_i) ( \mathrm{ind}_{\rU'}^{\rL'} \rO_{\mu} ) = \mathrm{ind}_{\rU'}^{\rL'} ( ad(\alpha_i) ( \rO_{\mu}) )$. 
Notice that after fixing one character of $\rU'$, all its $W(k)[\rU']$-lattices are equivalent, so $ad(\alpha_i) ( \cY_{\rL', \mu} )= \cY_{\rL', ad(\alpha_i) (\mu)}$. Hence, let $[\mu]$ denote the $\rT'$-orbit of $\mu$ in $\rO_{\mathrm{U}}(\rL')$, we have
$$\mathrm{Stab_{\rT}([\mu])} \subset \mathrm{Stab}_{\rT}(\cY_{\rL',\mu}) \subset \mathrm{Stab}_{\rT}(\cY_{\rL', \mu} \otimes \overline{\cK}),$$
where $\mathrm{Stab}$ denotes the group of stabiliser. On the other hand, the proof of lemma 2.3 a) in \cite{DiFl} tells that 
$$\mathrm{Stab}_{\rT}(\cY_{\rL', \mu} \otimes \overline{\cK}) \subset \mathrm{Stab_{\rT}([\mu])}.$$
So the inclusion above is in fact a bijection.
Combine this with the statement of lemma 2.3 a) in \cite{DiFl}, we finish our proof.
\end{proof}

\begin{prop}
\label{lemma a.1}
Fix a semisimple $\ell$-regular $s \in \rG'^{\ast}$, define $\mathcal{S}_{[s]}$ to be the set of semisimple $\ell$-regular $\widetilde{\rG}^{\ast}$-conjugacy classes $[\tilde{s}] \subset \widetilde{\rG}^{\ast}$ such that $i^{\ast}[\tilde{s}]=[s]$. Then
$$\bigoplus_{[\tilde{s}] \in \mathcal{S}_{[s]}} \mathrm{res}_{\rL'}^{\rL} \cY_{\rL,\tilde{s}}= \vert \rZ(\rL) : \rZ(\rL') \vert \bigoplus_{\mu \in \rO_{\rU'}(\rL')} \cY_{\rL',\mu,s} $$
\end{prop}

\begin{proof}
By definition that $\cY_{\rL, \tilde{s}}= b_{\tilde{s}} \cdot \cY_{\rL}$. Multiplying $b_s$ on both sides of the equation in Proposition \ref{proposition 3}  and considering Corollary \ref{corollary 3.1}, we conclude that for any $\ell$-regular semisimple $\rG'^{\ast}$-conjugacy class $[s]$, $\bigoplus_{[\tilde{s}] \in \mathcal{S}_{[s]}} \mathrm{res}_{\rL'}^{\rL} \cY_{\rL,\tilde{s}}$ is a projective direct summand of $\vert \rZ(\rL) : \rZ(\rL') \vert \bigoplus_{\mu \in \rO_{\rU}(\rL')} \cY_{\rL',\mu,s}$. Meanwhile, let $\mathcal{S}=\{ \mathcal{S}_{[s]}\vert\text{ } s\in\rG^{'\ast}, s\text{ semisimlpe }\ell\text{-regular} \}$, then Proposition \ref{proposition 3} can be written as:
$$\bigoplus_{\mathcal{S}_{[s]}\in \mathcal{S}}\bigoplus_{[\tilde{s}] \in \mathcal{S}_{[s]}} \mathrm{Res}_{\rL'}^{\rL} \cY_{\rL,\tilde{s}}= \vert \rZ(\rL) : \rZ(\rL') \vert \bigoplus_{[s]} \bigoplus_{\mu \in \rO_{\rU}(\rL')} \cY_{\rL',\mu,s} $$
Since there is a natural bijection between $\mathcal{S}$ and the set of semisimple conjugacy classes $\{ [ s ] \}$ in $\rG'^{\ast}$, we deduce the equation desired.
\end{proof}

\subsection{Uniqueness of supercuspidal support}
In this part, we will prove the main theorem \ref{theorem 1} for this section. 



\begin{thm} \label{theorem 1}
Let $\rL'$ be a standard split Levi subgroup of $\rG'$ and $\nu$ be a cuspidal $k$-representation of $\rL'$.Then the supercuspidal support of $\nu$ is unique up to $\rL'$-conjugation.
\end{thm}

Let $\mathrm{P}_{\nu}$ denote the $W(k)[\rL']$-projective cover of $\nu$. To prove the theorem above, we will follow the strategy below:

\begin{enumerate}
\item For any standard Levi subgroup $\rM'$ of $\rL'$, prove that $r_{\rM'}^{\rL'} \mathrm{P}_{\nu}$ is either equal to $0$ or indecomposable.
\item Prove that there is only one unique standard split Levi subgroup $\rM'$ of $\rL'$, such that $r_{\rM'}^{\rL'} \mathrm{P}_{\nu}$ is cuspidal.
\end{enumerate}
Let $(\rM', \theta)$ be a supercuspidal $k$-pair of $\rL'$. From the proof of Proposition 3.2 of \cite{Hiss}, we have that $(\rM', \theta)$ belongs to the supercuspidal support of $(\rL', \nu)$, if and only if $\mathrm{Hom}(r_{\rM'}^{\rL'} \mathrm{P}_{\nu}, \theta) \neq 0$. Combining this fact with step $1$ as above, we obtain that $r_{\rM'}^{\rL'}\mathrm{P}_{\nu}$ is the $W(k)[\rM']$-projective cover of $\theta$. Proposition $2.3$ of \cite{Hiss} states that an irreducible $k$-representation of $\rM'$ is supercuspidal if and only if its projective cover is cuspidal, hence Theorem \ref{theorem 1} is equivalent to step $2$.

\begin{rem}
\label{remark 2.11}
\begin{itemize}
\item The discussion above is true as well for Levi subgroups $\rL$ of $\rG$.
\item Proposition 3.2 of \cite{Hiss} concerns $k[\rL']$-projective cover, but from Proposition 42 of \cite{Serre} we know that there is a surjective morphism of $k[\rL']$-modules from the $W(k)[\rL']$-projective cover to the $k[\rL']$-projective cover, and hence obtain the same result for $W(k)[\rL']$-projective cover.
\end{itemize}
\end{rem}

\begin{prop}
\label{prop a}
Let $\nu$ be an irreducible cuspidal $k$-representation of $\rL'$. There exists a simple $k[\rL]$-module $\widetilde{\nu}$, and a surjective morphism $\mathrm{res}_{\rL'}^{\rL} \widetilde{\nu} \twoheadrightarrow \nu $. Furthermore, let $\cY_{\rL, \tilde{s}}$ be the $W(k)[\rL]$-projective cover of $\widetilde{\nu}$, where $\tilde{s} \in
\rG^{\ast}$ is an $\ell$-regular semisimple element, then there exists $\mu \in \rO_{\rU'}(\rL')$ such that $\cY_{\rL',\mu,s}$ is the $W(k)[\rL']$-projective cover of $\nu$.
\end{prop} 

\begin{proof}
By using Mackey formula to $\ind_{\rL'}^{\rL}\nu$, we can find such $\widetilde{\nu}$.

Since the restriction functor respects projectivity, we deduce the second part from Corollary \ref{corollary 3.4} and Proposition \ref{lemma a.1}.
\end{proof}


Let $\rM'$ be a standard split Levi subgroup of $\rL'$. It is clear that $\mu \vert_{\rM'}$ belongs to $\rO_{\rU'}(\rM')$. Now consider the intersection $[s] \cap \rM'^{\ast}$. As in the paragraph above Proposition 5.10 of \cite{Helm}, $[\tilde{s}] \cap \rM^{\ast}$ consists of one $\rM^{\ast}$-conjugacy class or is empty, so does $[s] \cap \rM'^{\ast}$. For the first case, $\cY_{\rM', \mu \vert_{\rM'}, [s] \cap \rM'^{\ast}}$ is well defined, and for the second case, we define it to be 0. From now on, we will always use $\cY_{\rM', \mu, s}$ to simplify $\cY_{\rM', \mu \vert_{\rM'}, [s] \cap \rM'^{\ast}}$. We use the same manner to define $\cY_{\rM, \tilde{s}}$.

\begin{prop}
\label{prop b}
Let $\nu$ be an irreducible cuspidal $k[\rL']$-representation, and $\widetilde{\nu}$, $\cY_{\rL', \mu, s}$, $\cY_{\rL, \tilde{s}}$ be as in Proposition $\ref{prop a}$. Then $r_{\rM'}^{\rL} \cY_{\rL', \mu, s}$ is equal to $0$ or indecomposable and isomorphic to $\cY_{\rM', \mu, s}$ as $W(k)[\rM']$-module.
\end{prop}

\begin{proof}
In the proof of Proposition \ref{lemma a.1} we know that $\cY_{\rL',\mu,s}$ is a direct summand of $\mathrm{Res}_{\rL'}^{\rL}(\cY_{\rL,\tilde{s}})$. Observing that the unipotent radical of $\rM'$ is also the unipotent radical of $\rM$, we deduce directly from the definition that $r_{\rM'}^{\rL'} (\mathrm{res}_{\rL'}^{\rL}(\cY_{\rL,\tilde{s}})) = \mathrm{res}_{\rM'}^{\rM} (r_{\rM}^{\rL}(\cY_{\rL,\tilde{s}}))$, and Proposition 5.10 in \cite{Helm} states that $r_{\rM}^{\rL} (\cY_{\rL,\tilde{s}})= \cY_{\rM,\tilde{s}}$. The statements above, combining with the fact that parabolic restriction is exact and respects projectivity, implies that  $r_{\rM'}^{\rL'} \cY_{\rL',\mu,s}$ is a projective direct summand of $\mathrm{res}_{\rM'}^{\rM}\cY_{\rM, \tilde{s}}$. Suppose $[\tilde{s}] \cap \rM^{\ast}$ is empty, then the same for $[s] \cap \rM'^{\ast}$, hence we have $\cY_{\rM, \tilde{s}}= r_{\rM'}^{\rL} \cY_{\rL',\mu,s}=0$.

Now consider the second case. Suppose that $[\tilde{s}] \cap \rM^{\ast}$ is non-empty, which is a $\rM^{\ast}$-conjugacy class $[\tilde{s}']$ for $\tilde{s}'\in\rM^{\ast}$. Let $\mu'$ be the non-degenarate character $\mathrm{res}_{\rU_{\rL'}}^{\rU_{\rM'}} \mu$, where $\rU_{\rL'}$ and $\rU_{\rM'}$ denote the unipotent radical of $\rL$ and $\rM$ respectively. Corollary 15.15 in [Bon] gives an equation:
$$r_{\rM'}^{\rL'} \cY_{\rL', \mu, s} \otimes \overline{\mathcal{K}}= \cY_{\rM',\mu',s'} \otimes \overline{\mathcal{K}}.$$
which means the $\ell$-regular quotient (see Definition \ref{defn fin aa}) of $r_{\rM'}^{\rL'} \cY_{\rL', \mu, s}$ is indecomposable, which is equivalent to say that $r_{\rM'}^{\rL'} \cY_{\rL', \mu, s}$ is indecomposable by \cite[lemma 5.11 ]{Geck}. Notice that Corollary 15.11 in [Bon] tells that the sub-representation of $\mathrm{Res}_{\rM'}^{\rM}\cY_{\rM, \tilde{s}} \otimes \overline{\mathcal{K}}$ corresponding to $[s']$ has multiplicity one, and the equation above says that the irreducible sub-representation corresponding to $[s']$ of $r_{\rM'}^{\rL'} \cY_{\rL', \mu, s} \otimes \overline{\mathcal{K}}$ and $\cY_{\rM',\mu',s'} \otimes \overline{\mathcal{K}}$ coincide, hence these two projective direct summands of $\mathrm{Res}_{\rM'}^{\rM}\cY_{\rM, \tilde{s}}$ coincide each other.

\end{proof}

We move on to the second step of Theorem \ref{theorem 1}. The statement of step 2 is true for $\rL$, hence there only left the proposition below to finish our proof:

\begin{prop}
Let $\cY_{\rL', \mu, s}$, $\cY_{\rL,\tilde{s}}$, $\widetilde{\nu}$ be as in Proposition \ref{prop a}, then for any standard split Levi $\rM'$ of $\rL'$, we have $r_{\rM'}^{\rL'} \cY_{\rL',\mu, s}= \cY_{\rM', \mu, s}$ is cuspidal if and only if $r_{\rM}^{\rL} \cY_{\rL, \tilde{s}}= \cY_{\rM, \tilde{s}}$ is cuspidal.
\end{prop}
 
\begin{proof}
The regular inclusion $i$ induces a bijection preserving partial order between standard Levi subgroups of $\rG$ and $\rG'$, the statement in the proposition is equivalent to say that for any split Levi $\rM'$ of $\rL'$, 
$$r_{\rM'}^{\rL'} \cY_{\rL',\mu, s} =0 \iff r_{\rM}^{\rL} \cY_{\rL, \tilde{s}}=0.$$
The proof of Proposition \ref{prop b} tells us
$$r_{\rM'}^{\rL'} \cY_{\rL',\mu, s} \hookrightarrow \mathrm{res}_{\rM'}^{\rM} \cY_{\rM,\tilde{s}},$$
hence the direction $\Rightarrow$ is clear.

Now consider the other direction. Notice that $r_{\rM'}^{\rL'} \cY_{\rL',\mu,s}$ is an $\rO[\rM']$-lattice, and definition 5.9 in \cite{Geck} tells us that $r_{\rM'}^{\rL'} \cY_{\rL',\mu, s}=0$ if and only if its $\ell$-regular quotient $\pi_{l'} (r_{\rM'}^{\rL'}\cY_{\rL', \mu, s})=0.$ 
By \cite[Corollary 15.15]{Bonn}, the $\overline{\mathcal{K}}[\rM']$-module $(\pi_{l'} (r_{\rM'}^{\rL'} \cY_{\rL', \mu, s})) \otimes \overline{\mathcal{K}}$ is the sum of irreducible $\overline{\mathcal{K}}[\rM']$-submodules of $\mathrm{ind}_{\rU'_{\rM'}}^{\rM'} \mu$ corresponding to $[s] \cap \rM'^{\ast}$, where $[s]$ denotes the $\rL'^{\ast}$-conjugacy class. Hence $(\pi_{l'} (r_{\rM'}^{\rL'} \cY_{\rL', \mu, s})) \otimes \overline{\mathcal{K}}=0$ implies $[s] \cap \rM'^{\ast}=0$, which means $[\tilde{s}] \cap \rM^{\ast}=0$, and $\cY_{\rM,{\tilde{s}}}=0$.
 
\end{proof}

\section{$k$-representations of $p$-adic groups $\mathrm{SL}_n(F)$}
\label{section 4}
Let $F$ be a non-archimedean locally compact field of residual characteristic $p$, which is different from $\ell$. Let $\rM$ be a Levi subgroup of $\rG=\mathrm{GL}_n(F)$, we always denote by $\rM'$ the Levi subgroup of $\rG'$ such that $\rM\cap\rG'=\rM'$. Let $\pi$ be an irreducible cuspidal $k$-representation of $\rM'$, by \cite[Proposition 2.2]{TA} there exists an irreducible $k$-representation $\pi$ of $\rM$ such that $\pi'$ appears as a direct component of $\pi\vert_{\rM'}$, which is semisimple with finite length (a same proof as in \cite{TA} can be generalised to the case when $\ell$ is positive as explained in \cite{C}). Furthermore, any such $\pi$ is cuspidal, which follows from the fact that the unipotent radical of a Parabolic subgroup of $\rG$ lies in the kernel of the determinant function. The supercuspidal support of $\pi$ is unique up to $\rM$-conjugation (\cite{V2}). 

We prove in Section \ref{section2.2}, that the supercuspidal support of $\pi'$ is also unique up to $\rM$-conjugation (Proposition \ref{Lproposition 10}), which is the first description of supercuspidal support. We will first generalise the definition of $n$-th derivative given by Bernstein and Zelevinsky in \cite{BeZe} for complex representations of $\GL_n(F)$ to $k$-representations of Levi subgroups $\rM'$ of $\SL_n(F)$, which gives a link between the higgest derivative of an irreducible representation and that of its supercuspidal support. Then we deduce the uniqueness of supercuspidal support of irreducible $k$-representations of $\rM'$ in Theorem \ref{Uthm 3}., based on the result of Proposition \ref{Lproposition 10}, by considering its Whittaker model of a fixed non-degenerate character.

\subsection{First description of supercuspidal support}
\label{section2.2}

\begin{lem}
\label{Llemma 12}
Let $\pi$ be an irreducible $k$-representation. If $\pi\otimes\chi\circ\det$ is supercuspidal for a $k$-quasicharacter $\chi$ of $F^{\times}$, then $\pi$ is supercuspidal.
\end{lem}

\begin{proof}
Assume that there is a supercuspidal representation $\tau$ of a proper Levi $\rL$ of $\rM$ such that $\pi$ is an irreducible subquotient of $i_{\rL}^{\rM}\tau$. Then $\pi\otimes\chi\circ\det$ is a subquotient of $i_{\rL}^{\rM}\tau\otimes\chi\circ\det$, which follows from the equivalence 
$$i_{\rL}^{\rM}\tau\otimes\chi\circ\det\cong (i_{\rL}^{\rM}\tau)\otimes\chi\circ\det.$$
The above equivalence is obtained from [\S I,5.2,d)]\cite{V1}, by noticing that for any parabolic subgroup containing $\rL$, its unipotent radical is a subset of the kernel of the determinant function.
\end{proof}

\begin{lem}
\label{Llemma 11}
Let $\pi'$ be an irreducible cuspidal $k$-representation of $\rM'$, and $\pi$ an irreducible $k$-representation of $\rM$ containing $\pi'$. Then $\pi'$ is supercuspidal if and only if $\pi$ is supercuspidal.
\end{lem}

\begin{proof}
Let $\rL$ be a Levi subgroup of $\rM$, and $\rL'=\rL\cap\rM'$. We have $\res_{\rL'}^{\rL}r_{\rL}^{\rM}\pi\cong r_{\rL'}^{\rM'}\res_{\rM'}^{\rM}\pi$, while $\res_{\rM'}^{\rM}\pi$ is a direct sum of $\rM$-conjugations of $\pi'$, hence the later one is zero, and we obtain that $\pi$ is cuspidal.

We assume that $\pi$ is non-supercuspidal, which means there exists a supercupidal representation $\tau$ of a proper Levi subgroup $\rL$ of $\rM$, the representation $\pi$ is a subquotient of the parabolic induction $i_{\rL}^{\rM}\tau$. Now by \S $5.2$ \cite{BeZe}, we obtain:
$$\res_{\rM'}^{\rM}i_{\rL}^{\rM}\tau\cong\ i_{\rL'}^{\rM'}\res_{\rL'}^{\rL}\tau.$$
There must be a direct component $\tau'$ of $\res_{\rL'}^{\rL}\tau$, and $\pi'$ be an irreducible subquotient of $i_{\rL'}^{\rM'}\tau'$. Hence $\pi'$ is not supercuspidal, which contradicts with the assumption.

\end{proof}

\cite[Corollary 2.5]{TA} can be generalised to $k$-representations. We write this proposition here for convinient

\begin{prop}
\label{Lproposition 9}
Let $\pi'$ be an irreducible cuspidal $k$-representation of $\rM'$. If $\pi_1,\pi_2$ two irreducible cuspidal $k$-representations of $\rM$, such that $\pi'$ appears as a direct component of $\res_{\rM'}^{\rM}\pi_1$ and $\res_{\rM'}^{\rM}\pi_2$ in common, then there exists a $k$-quasicharacter of $F^{\times}$ verifying that $\pi_1\cong \pi_2\otimes\chi\circ\det$.
\end{prop}

\begin{prop}
\label{Lproposition 10}
Let $\pi'$ be an irreducible cuspidal $k$-representation of $\rM'$, and $\pi$ an irreducible cuspidal $k$-representation of $\rM$ such that $\pi$ contains $\pi'$. Let $[\rL,\tau]$ be the supercuspidal support of $\pi$, where $\rL$ is a Levi subgroup of $\rM$ and $\tau$ an irreducible supercuspidal $k$-representation of $\rL$. Let $\tau'$ be an arbitrary direct component of $\res_{\rL'}^{\rL}\tau$. A supercuspidal pair belonging to the supercuspidal support of $\pi'$ is $\rM$-conjugated to $(\rL',\tau')$.
\end{prop}

\begin{proof}
Let $\rL_0'$ be a Levi subgroup of $\rM'$ and $\tau_0'$ an irreducible supercuspidal $k$-representation of $\rL_0'$. Let $\tau_0$ be an irreducible $k$-representation of $\rL_0$ containing $\tau_0'$, hence $\tau_0$ is supercuspidal as in Lemma \ref{Llemma 11}.

Now suppose that $\pi'$ is an irreducible subquotient of $i_{\rL_0'}^{\rM'}\tau_0'$. 
 By the same reason as in the proof of Lemma \ref{Llemma 11}, we know that there must be an irreducible subquotient of $i_{\rL_0}^{\rM}\tau_0$, noted as $\pi_0$, such that $\pi'$ is a direct component of $\res_{\rM'}^{\rM}\pi_0$. From \cite[Corollary 2.5]{TA}, there exists a $k$-quasicharacter $\chi$ of $F^{\times}$ such that $\pi_0\cong\pi\otimes\chi\circ\det$. On the other hand, the supercuspidal support of $\pi\otimes\chi\circ\det$ is the $\rM$-conjugacy class of $(\rL,\tau\otimes\chi\circ\det)$. Hence we may assume that $\rL_0=\rL$ and $\tau_0\cong\tau\otimes\chi\circ\det$, and deduce that $\tau_0'$ is a direct component of $\res_{\rL'}^{\rL}\tau\otimes\chi\circ\det\cong\res_{\rL'}^{\rL}\tau$.
\end{proof}

\subsection{The $n$-th derivative and parabolic induction}
\label{subsubsection 5.1.1}
This section is a direct generalisation of the part of derivatives given in \cite{BeZe} for $\mathrm{GL}_n$. Neverthless, except the convenience reason, we write this section because the author believe the notation system for $\rG'$ is worthy to be introduced, which is different and complicate compared to that of $\mathrm{GL}_n$. The complexity arises from the fact that the Levi subgroups $\rM'$ can not be written in the form of a product of $\mathrm{SL}$ groups in lower rank, so a method of recursion can not be applied here.

Let $n_1,\ldots,n_m$ be a family of integers, and let $\rM_{n_1,\ldots,n_m}$ be the product $\mathrm{GL}_{n_1}\times\cdots\times\mathrm{GL}_{n_m}$, which can be canonically embedded into $\mathrm{GL}_{n_1+\cdots+n_m}$. Let $\rM_{n_1,\ldots,n_m}'$ be $\rM_{n_1,\ldots,n_m}\cap\mathrm{SL}_{n_1+\cdots+n_m}$, and $P_{n_1}$ the mirabolic subgroup of $\mathrm{GL}_{n_1}$. For any $i\in\{1,\ldots,m\}$, let $\mathrm{U}_{n_i}$ be the subset of $\mathrm{GL}_{n_i}$, consisted with upper-triangular matrix with $1$ on the diagonal. We denote $\rU_{n_1,\ldots,n_m}=\rU_{n_1}\times\cdots\times\rU_{n_m}$ by  $\rU_{n_1,\ldots,n_m}$.

\begin{defn}
\label{ddef 1}
Let $n_1,\ldots,n_m$ be a family of positive integers, and $s\in\{1,\ldots,m\}$. We define:
\begin{itemize}
\item the mirabolic subgroup at place $s$ of $\rM_{n_1,\ldots,n_m}$, as $P_{(n_1,\ldots,n_m),s}=\mathrm{GL}_{n_1}\times\cdots\times\mathrm{GL}_{n_{s-1}}\times P_{n_s}\times\mathrm{GL}_{n_{s+1}}\times\cdots\times\mathrm{GL}_{n_m}$;
\item the mirabolic subgroup at place $s$ of $\rM_{n_1,\ldots,n_m}'$, as $P_{(n_1,\ldots,n_m),s}'=\mathrm{GL}_{n_1}\times\cdots\times\mathrm{GL}_{n_{s-1}}\times P_{n_s}\times\mathrm{GL}_{n_{s+1}}\times\cdots\times\mathrm{GL}_{n_m}\cap\rM_{n_1,\ldots,n_m}'$.
\end{itemize}
\end{defn}

We fix $\theta_{i}$ a non-degenerate character of $\rU_{n_i}$. It is clear that $\rU_{n_1,\ldots,n_m}$ is a subgroup of $P_{(n_1,\ldots,n_m),s}$ and $P_{(n_1,\ldots,n_m),s}'$ for any $s\in\{1,\ldots,m\}$. Let $V_{n_s-1}$ be the additive group of $k$-vector space with dimension $n_s-1$, which can be embedded canonically as a normal subgroup in $\rU_{n_1}\times\cdots\times\rU_{n_m}$. The subgroup $V_{n_s-1}$ is normal both in $P_{(n_1,\ldots,n_m),s}$ and $P_{(n_1,\ldots,n_m),s}'$, furthermore, we have $P_{(n_1,\ldots,n_m),s}=\rM_{n_1,\ldots,n_s-1,\ldots,n_m}\cdot V_{n_s-1}$ and $P_{(n_1,\ldots,n_m),s}'=\rM_{n_1,\ldots,n_s-1,\ldots,n_m}'\cdot V_{n_s-1}$. Under the notation system as explained above, the notation $\rM_{0}$ is meaningful in this section, which denotes a subgroup of $P_{(1),1}$, hence is the trivial group.

Denote $\theta$ a $k$-character of $\rU_{n_1}\times\cdots\times\rU_{n_m}$. For any $k$-representation $(E,\rho)$ inside $\mathrm{Rep}_k(P_{(n_1,\ldots,n_m),s}')$, where $E$ is the representation space of $\rho$, let $E_{s,\theta}$ denote the subspace of $E$ generated by elements in form of $\rho(g)a-\theta(g)a$, where $g\in V_{n_s-1},a\in E$. We define the coinvariants of $(E,\rho)$ according to $\theta$ as $E/E_{s,\theta}$, and denote it as $E(\theta,s)$, and view $E(\theta,s)$ as a $k$-representation of $\rM_{n_1,\ldots,n_s-1,\ldots,n_m}'$.

\begin{defn}
\label{ddef 2}
Fix a non-degenerate character $\theta$ of $\rU_{n_1}\times\cdots\times\rU_{n_m}$. 
\begin{itemize}
\item Let $(E,\rho)\in\mathrm{Rep}_k(P_{(n_1,\ldots,n_m),s}')$, define
$$\Psi_{s}^{-}:\mathrm{Rep}_k(P_{(n_1,\ldots,n_m),s}')\rightarrow\mathrm{Rep}_k(\rM_{n_1,\ldots,n_s-1,\ldots,n_m}'),$$
the canonical projection from $E$ to $E(\mathds{1},s)$;
\item Let $(E,\rho)\in\mathrm{Rep}_k(\rM_{n_1,\ldots,n_s-1,\ldots,n_m}')$, and write $P_{(n_1,\ldots,n_m),s}'=\rM_{n_1,\ldots,n_s-1,\ldots,n_m}' \cdot V_{n_s-1}$. Define
$$\Psi_{s}^{+}:\mathrm{Rep}_k(\rM_{n_1,\ldots,n_s-1,\ldots,n_m}')\rightarrow\mathrm{Rep}_k(P_{(n_1,\ldots,n_m),s}'),$$
which maps $(E,\rho)$ to $(E,\Psi^{+}_{s}(\rho))$ by $\Psi_{s}^{+}(\rho)(mg)(a)=\rho(m)(a)$, for any $m\in\rM_{n_1,\ldots,n_s-1,\ldots,n_m}', g\in V_{n_s-1}$ and $a\in E$. 
\item Let $(E,\rho)\in\mathrm{Rep}_k(P_{(n_1,\ldots,n_m),s}')$, define
$$\Phi_{\theta,s}^{-}:\mathrm{Rep}_k(P_{(n_1,\ldots,n_m),s}')\rightarrow\mathrm{Rep}_k(P_{(n_1,\ldots,n_s-1,\ldots,n_m),s}'),$$
which maps $E$ to $E(\theta,s)$, and consider the restricted representation from $\rM_{n_1,\ldots,n_s-1,\ldots,n_m}'$ to $P_{(n_1,\ldots,n_s-1,\ldots,n_m),s}'$;
\item Let $(E,\rho)\in\mathrm{Rep}_k(P_{(n_1,\ldots,n_s-1,\ldots,n_m),s}')$. Consider the composed canonical embedding 
$$ P_{(n_1,\ldots,n_s-1,\ldots,n_m),s}' \hookrightarrow \rM_{(n_1,\ldots,n_s-1,\ldots,n_m)} \hookrightarrow P_{(n_1,\ldots,n_m),s}',$$
and denote $(E,\rho_{\theta})$ the $k$-representation of $P_{(n_1,\ldots,n_s-1,\ldots,n_m),s}'\cdot V_{n_s-1}$ under the above embedding, such that $\rho_{\theta}(pg)=\theta(g)\rho(p)$, for any $p\in P_{(n_1,\ldots,n_s-1,\ldots,n_m),s}',g\in V_{n_s-1}$. Define
$$\Phi_{\theta,s}^{+}:\mathrm{Rep}_k(P_{(n_1,\ldots,n_s-1,\ldots,n_m),s}')\rightarrow\mathrm{Rep}_k(P_{(n_1,\ldots,n_m),s}'),$$
by taking $\Phi_{\theta,s}^{+}(\rho)=\ind_{P_{(n_1,\ldots,n_s-1,\ldots,n_m),s}'\cdot V_{n_s-1}}^{P_{(n_1,\ldots,n_m),s}'}\rho_{\theta}$.
\end{itemize}
\end{defn}

\begin{rem}
\label{drem 3}
By the reason that for any $m\in\mathbb{Z}$ the group $V_m$ is a limite of pro-$p$ open compact subgroups, the four functors defined above are exact.
\end{rem}

The notion of derivatives is well defined for $k$-representations of $\rG$, now we consider the parallel operator of derivatives for Levi subgroups of $\rG'$.
\begin{defn}
\label{ddef 4}
Fix a non-degenerate character $\theta$ of $\rU_{n_1}\times\cdots\times\rU_{n_m}$. Let $(E,\rho)\in\mathrm{Rep}_k(P_{(n_1,\ldots,n_m,s)}')$, for any interger $s\in\{1,\ldots,m\}$ and $1\leq d\leq n_1+\ldots+n_s$, we define the derivative $\rho_{\theta,s}^{(d)}$:
\begin{itemize}
\item when $1\leq d\leq n_s$, define 
$$\rho_{\theta,s}^{(d)}=\Psi_s^{-}\circ(\Phi_{\theta,s}^-)^{d-1}\rho: \mathrm{Rep}_k(P_{(n_1,\ldots,n_m),s}')\rightarrow \mathrm{Rep}_k(\rM_{(n_1,\ldots,n_s-d,\ldots,n_m),s}');$$
\item when $n_s+1\leq d=n_s+\ldots+n_{s-l}+n'$, where $0\leq l\leq s-1$ and $1\leq n'\leq n_{s-l-1}$, then $\rho_{\theta,s}^{(d)}=\Psi_{s-l-1}^{-}\circ(\Phi_{\theta,s-l-1})^{n'-1}\circ(\Phi_{\theta,s-l})^{n_{s-l}-1}\circ\ldots\circ(\Phi_{\theta,s}^-)^{n_s-1}\rho$, and
$$\rho_{\theta,s}^{(d)}: \mathrm{Rep}_k(P_{(n_1,\ldots,n_m),s}')\rightarrow \mathrm{Rep}_k(\rM_{n_1,\ldots,n_{s-l-1}-n',n_{s+1},\ldots,n_m}).$$
\end{itemize}
\end{defn}

\begin{defn}
\label{ddef 5}
Suppose $m$ is bigger than $2$. To simplify our notations, we introduce $\ind_{m}^{m-1}:\mathrm{Rep}_{k}(\rG_1)\rightarrow\mathrm{Rep}_k(\rG_2)$ according to different cases:
\begin{itemize}
\item When $\rG_1=\rM_{n_1,\ldots,n_m}'$ and $\rG_2=\rM_{n_1,\ldots,n_{m-1}+n_m}'$, we embed $\rG_1$ into $\rG_2$ as in the figure case $\mathrm{I}$, and $\ind_{m}^{m-1}$ is defined as $i_{\rU,1}$, and the later one is defined as in \S$1.8$ of \cite{BeZe};
\item When $\rG_1=P_{(n_1,\ldots,n_m),m}'$ and $\rG_2=P_{(n_1,\ldots,n_{m-1}+n_m),m-1}'$, we embed $\rG_1$ into $\rG_2$ as in the figure case $\mathrm{II}$, and $\ind_{m}^{m-1}$ is defined as $i_{\rU,1}$;
\item When $\rG_1=P_{(n_1,\ldots,n_m),m-1}'$ and $\rG_2=P_{(n_1,\ldots,n_{m-1}+n_m),m-1}'$, we embed $\rG_1$ into $\rG_2$ as in the figure case $\mathrm{III}$, and $\ind_{m}^{m-1}$ is defined as $i_{\rU,1}\circ\varepsilon$. Here $\varepsilon$ is a character of $P_{(n_1,\ldots,n_m),m-1}'$. Write $g\in P_{(n_1,\ldots,n_m),m-1}'\subset\rM_{n_1,\ldots,n_m}'$ as $(g_1,\ldots,g_m)$, define $\varepsilon(g)=\vert \det(g_m) \vert$, the absolute value of $\det(g_m)$, which is a power of $p$. This $k$-character is well defined since $p\neq l$.
\end{itemize}
\end{defn}

\begin{figure}
\centering
\begin{tikzpicture}
\draw[fill=gray] (0,1) rectangle (1,0); 
\draw[fill=gray] (1,0) rectangle (3,-2);
\draw[fill=gray] (3,-2) rectangle (5,-4);
\draw[fill=white] (1,-2) rectangle (3,-4) ;
\draw [fill=white] (3,0) rectangle (5,-2);
\draw[decorate,decoration={brace,amplitude=7pt}]  (1,-2)--(1,0);
\draw[decorate,decoration={brace,amplitude=7pt}]  (1,-4)--(1,-2);
\node at (0.3,-1){$n_{m-1}$};
\node at (0.3,-3){$n_m$};
\node at (2,-3){$0$};
\node at (4,-1){$U$};
\end{tikzpicture}
\caption{Case I}

\centering
\begin{tikzpicture}
\draw[fill=gray] (0,1) rectangle (1,0);
\draw[fill=gray] (1,0) rectangle (3,-2);
\draw[fill=white] (3,0) rectangle (5,-2);
\draw[fill=white] (1,-2) rectangle (3,-4);
\draw[fill=gray] (3,-2) rectangle (5,-3.5);
\draw[fill=white] (3,-3.5) rectangle (4.5,-4);
\draw[fill=gray] (4.5,-3.5) rectangle (5,-4);
\draw[decorate,decoration={brace,amplitude=7pt}]  (1,-2)--(1,0);
\draw[decorate,decoration={brace,amplitude=7pt}]  (5,-2)--(5,-3.5);
\draw[decorate,decoration={brace,amplitude=3pt}]  (5,-3.5)--(5,-4);
\node at (0.3,-1){$n_{m-1}$};
\node at (5.85,-2.75) {$n_m-1$};
\node at (5.25,-3.75) {$1$};
\node at (2,-3){$0$};
\node at (4,-1){$U$};
\node at (4.75,-3.75) {$1$};
\node at (3.75,-3.75) {$0$};
\end{tikzpicture}
\caption{Case II}
\end{figure}

\begin{figure}
\centering
\begin{tikzpicture}
\draw[fill=gray] (0,1) rectangle (1,0); 
\draw[fill=gray] (1,0) rectangle (2.5,-1.5);
\draw[fill=gray] (2.5,-1.5) rectangle (4.5,-3.5);
\draw[fill=gray] (4.5,-3.5) rectangle (5,-4);
\draw[fill=white] (1,-1.5) rectangle (2.5,-3.5);
\draw[fill=white] (2.5,0) rectangle (4.5,-1.5);
\draw[fill=gray] (4.5,0) rectangle (5,-1.5);
\draw[fill=white] (4.5,-1.5) rectangle (5,-3.5);
\draw[fill=white] (1,-3.5) rectangle (2.5,-4);
\draw[fill=white] (2.5,-3.5) rectangle (4.5,-4);
\draw[decorate,decoration={brace,amplitude=7pt}]  (1,-1.5)--(1,0);
\draw[decorate,decoration={brace,amplitude=7pt}]  (5,-1.5)--(5,-3.5);
\draw[decorate,decoration={brace,amplitude=3pt}]  (5,-3.5)--(5,-4);
\node at (0.01,-0.75) {$n_{m-1}-1$};
\node at (5.6,-2.5) {$n_m$};
\node at (5.25,-3.75) {$1$};
\node at (4.75,-3.75) {1};
\node at (3.5,-0.75) {$U$};
\node at (1.75,-2.5) {$0$};
\node at (1.75,-3.75) {$0$};
\node at (3.5,-3.75) {$0$};
\node at (4.75,-2.5) {$0$};
\end{tikzpicture}
\caption{Case III}
\end{figure}

\begin{prop}
\label{dprop 5}
Assume that $\rho_1\in\mathrm{Rep}_k(\rM_{n_1,\ldots,n_m}'), \rho_2\in\mathrm{Rep}_k(P_{(n_1,\ldots,n_m),m}')$, and $\rho_3\in\mathrm{Rep}_k(P_{(n_1,\ldots,n_m),m-1}')$. The functor $\ind_{m}^{m-1}$ is defined as in \ref{ddef 5} according to different cases, and we have the following properties:
\begin{enumerate}
\item 
In $\mathrm{Rep}_k(P_{(n_1,\ldots,n_{m-1}+n_m),m-1}')$, there exists an exact sequence:
$$0\rightarrow\ind_m^{m-1}(\rho_1\vert_{P_{m,m-1}'})\rightarrow(\ind_m^{m-1}\rho_1)\vert_{P_{m-1,m-1}'}\rightarrow \ind_m^{m-1}(\rho_1\vert_{P_{m,m}'})\rightarrow 0,$$
where $P_{m,m-1}'$ denotes $P_{(n_1,\ldots,n_m),m-1}'$, $P_{m-1,m-1}'$ denotes $P_{(n_1,\ldots,n_{m-1}+n_m),m-1}'$, and $P_{m,m}'$ denotes $P_{(n_1,\ldots,n_m),m}'$.
\item When $2\leq m$, let $\dot{\theta}$ be a non-degenerate character of $\rU_{n_1}\times\cdots\times\rU_{n_{m-2}}\times\rU_{n_{m-1}+n_m}$, such that $\dot{\theta}\vert_{U_{n_1}\times\cdots\times \rU_{n_m}}\cong\theta$. We have equivalences:
\begin{itemize}
\item $\ind_{m}^{m-1}\circ\Psi_m^{-}(\rho_2)\cong\Psi_{m-1}^{-}\circ\ind_{m}^{m-1}(\rho_2);$
\item $\ind_{m}^{m-1}\circ\Phi_{\theta,m}^{-}(\rho_2)\cong\Phi_{\dot{\theta},m-1}^{-}\circ\ind_{m}^{m-1}(\rho_2).$
\end{itemize}
\item We have an equivalence:
$$\Psi_{m-1}^{-}\circ\ind_{m}^{m-1}(\rho_3)\cong\ind_{m}^{m-1}\circ\Psi_{m-1}^{-}(\rho_3),$$
and an exact sequence:
$$0\rightarrow\ind_{m}^{m-1}\circ\Phi_{\theta,m-1}^{-}(\rho_3)\rightarrow\Phi_{\dot{\theta},m-1}^-\circ\ind_{m}^{m-1}(\rho_3)\rightarrow\ind_{m}^{m-1}((\Psi_{m-1}^{-}\rho_3)\vert_{P'})\rightarrow 0,$$
where $P'=P_{(n_1,\ldots,n_{m-1}-1,n_m),m}'$.
\end{enumerate}
\end{prop}

\begin{proof}
As proved in the Appendix of \cite{C}, Theorem 5.2 in \cite{BeZe} holds for $k$-representations of $\rM'$. Now let $n=n_1+\ldots+n_m$.

For $(1)$: Let $\rM'=\rM_{n_1,\ldots,n_m}'$ be embedded into $\rG'=\rM_{n_1,\ldots,n_{m-1}+n_m}'$ as in definition \ref{ddef 5}, figure $\mathrm{I}$. Define functor $\mathrm{F}$ as $\mathrm{F}(\rho_1)=\rho_1\vert_{P_{(n_1,\ldots,n_{m-1}+n_m),m-1}'}$, where the functor $\mathrm{F}$ is equivalently defined as in \S 5.1 \cite{BeZe} under the following setting (the notation $\rM$ in \cite{BeZe} corresponds $\rM'$ here): 
$$\rU \text{ as in figure }\mathrm{I},\vartheta=1, \rP=\rM'\rU;$$
$$\rN=P_{(n_1,\ldots,n_{m-1}+n_m),m-1}',\rV=\{ e\},\mathrm{Q}=\rN\rV.$$
To compute $\mathrm{F}$, we apply theorem 5.2 \cite{BeZe}. Condition $(1),(2)$ and $(\ast)$ from 5.1 \cite{BeZe} hold trivially. Let $\rT$ be the group of diagonal matrix, the $\mathrm{Q}$-orbits on $X=\mathrm{P}\backslash \rG$ is actually the $\rT\rN$-orbits, and the group $\rT\rN$ is a parabolic subgroup. By Bruhat decomposition $\rT\rN$ has two orbits: the closed orbit $Z$ of point $\mathrm{P}\cdot e\in X$ and the open orbit $Y$ of the point $\mathrm{P}\cdot\omega^{-1}\in X$, where $\omega$ is the matrix of the cyclic permutation $sgn(\sigma)\mathds{1}_{n_m}\cdot\sigma$. Here $\sigma$ is a permutation
$$\sigma=(n_1+\cdots+n_{m-1}\rightarrow n\rightarrow n-1\rightarrow\cdots\rightarrow n_1+\cdots+n_{m-1}),$$
and $sgn(\sigma)$ denotes the signal of $\sigma$, and $sgn(\sigma)\mathds{1}_{n_m}$ denotes an element in $\rM_{n_1,\ldots,n_m}'$, which is equal to identity on the first $m-1$ blocs, and a scalar matrix with value of $sgn(\sigma)$ on the last bloc. Now we consider condition $(4)$ from 5.1 \cite{BeZe}:
\begin{itemize}
\item Since $V=\{ e\}$, it is clear that $\omega(\mathrm{P}),\omega(\rM')$ and $\omega(U)$ are decomposable with respect to $(\rN,\rV)$;
\item Let us consider $\omega^{-1}(\mathrm{Q})=\omega^{-1}(\rN)$.
\end{itemize}
To study the intersection $\omega^{-1}(\rN)\cap (\rM\cdot\rU)$, first we consider a Levi subgroup $\rM_{n_1,\ldots,n_{m-1}+n_m-1,1}'$ of $\rG'$ and the corresponding standard parabolic subgroup 
$$\mathrm{P}'=\rM_{n_1,\ldots,n_{m-1}+n_m-1,1}'\cdot \rV_{n_{m-1}+n_m-1},$$
where $\rV_{n_{m-1}+n_m-1}$ denotes the unipotent radical of $\mathrm{P}'$. We have $\mathrm{N}\subset \mathrm{P}'$, hence $\omega^{-1}(\mathrm{N})\subset\omega^{-1}(\mathrm{P}')$. As in 6.1 of \cite{BeZe}, after fix a system $\Omega$ of roots, and denote $\Omega^{+}$ the set of positive roots. Then by Proposition in 6.2 \cite{BeZe}, we can write $\omega^{-1}(\mathrm{P}')=\rG(\mathcal{S})$ and $\mathrm{P}=\rG(\mathcal{P}),\rU=\rU(\mathcal{M})$ in the manner as in 6.1\cite{BeZe}, where $\mathcal{S},\mathcal{P}$ and $\mathcal{M}$ are convex subset of $\Omega$. So by a same computation as in Proposition in 6.1 \cite{BeZe}, we have:
$$\omega^{-1}(\mathrm{P}')\cap\mathrm{P}=(\omega^{-1}(\mathrm{P}')\cap\rM')\cdot(\omega^{-1}(\mathrm{P}')\cap\rU).$$
Notice that $\omega^{-1}(\mathrm{P}')\cap\rU=\omega^{-1}(\rN)\cap\rU$, we deduce that:
$$\omega^{-1}(\rN)\cap\mathrm{P}=(\omega^{-1}(\rN)\cap\rM')\cdot(\omega^{-1}(\rN)\cap\rU).$$
In the formula of $\Phi_{Z}$ in 5.2 \cite{BeZe}, since $\rU\cap\omega^{-1}(\rN)=\rU$, the characters $\varepsilon_1=\varepsilon_2=1$. Hence we obtain the exact sequence desired.

For $(2)$. In this part, the functor $\ind_{m}^{m-1}$ was defined differently according to different cases of Definition \ref{ddef 5}. First we consider the first part concerning about $\Psi_m^-$. Define functor $\mathrm{F}$ as $\Psi_{m-1}^-\circ\ind_{m}^{m-1}$. We write $\mathrm{F}$ as in \S $5.1$ \cite{BeZe} in the situation:
$$\rG=P_{(n_1,\ldots,n_{m-1}+n_m),m-1}',\rM=P_{(n_1,\ldots,n_m),m}', \rU \text{ are defined in Definition } \ref{ddef 5} \text{ case } \mathrm{II},$$
$$\rN=\rM_{n_1,\ldots,n_{m-1}+n_m-1}',\mathrm{V}=\rV_{n_{m-1}+n_m-1}.$$
Condition $(1)$ and $(2)$ of \S 5.1 \cite{BeZe} is clear. Since $\mathrm{Q}=\rG$, and there is only one $\mathrm{Q}$-orbit on $X=\mathrm{P} \backslash\rG$, conditions $(3),(4)$ hold trivially. Thus we obtain the equivalence:
$$\ind_{m}^{m-1}\circ\Psi_m^{-}\rho_2\cong\Psi_{m-1}^{-}\circ\ind_{m}^{m-1}\rho_2.$$

For the second part concerning about $\Phi_{\dot{\theta},m-1}^-$. Define functor $\mathrm{F}$ as $\Phi_{\dot{\theta},m-1}^-\circ\ind_{m}^{m-1}$. We write $\mathrm{F}$ as in \S $5.1$ \cite{BeZe} in the situation:
$$\rG=P_{(n_1,\ldots,n_{m-1}+n_m),m-1}',$$
$$\rN=P_{(n_1,\ldots,n_{m-1}+n_m-1),m-1}',\mathrm{V}=\mathrm{V}_{n_{m-1}+n_m-1},$$
and $\rM,\rU$ are defined as the case $\mathrm{II}$ of \ref{ddef 5}. Same orbits and same computation as in Proposition 4.13 $\mathrm{(b)}$ of \cite{BeZe} can be applied here.


For part $(3)$. Define functor $\mathrm{F}=\Psi_{m-1}^-\circ\ind_m^{m-1}$, we have (in the manner of \S $5.1$ \cite{BeZe}):
$$\rG=P_{(n_1,\ldots,n_m),m}',$$
$$\rN=\rM_{n_1,\ldots,n_m-1}',\mathrm{V}=\mathrm{V}_{n_m-1},$$
and $\rM,\rU$ as in \ref{ddef 5} case $\mathrm{II}$. There is only one $\mathrm{Q}$-orbit on $\mathrm{P}\backslash\rG$, and condition $(1)-(4)$ and $(\ast)$ in \S $5.1$ \cite{BeZe} hold. Notice that $\varepsilon\circ\Psi_{m-1}^-\cong\Psi_{m-1}^-$ ($\varepsilon$ is defined in Definition \ref{ddef 5}). After applying theorem $5.2$ of \cite{BeZe}, we obtain the equivalence:
$$\Psi_{m-1}^-\circ\ind_{m}^{m-1}\rho_3\cong\ind_{m}^{m-1}\circ\Psi_{m-1}^-\rho_3.$$

Define functor $\mathrm{F}=\Phi_{\dot{\theta},m-1}^-\circ\ind_m^{m-1}$. We have (in the manner of \S $5.1$ \cite{BeZe}):
$$\rG=P_{(n_1,\ldots,n_m),m}',$$
$$\rN=P_{(n_1,\ldots,n_{m-1}+n_m-1),m-1}',\mathrm{V}=\mathrm{V}_{n_{m-1}+n_m-1},$$
and $\rM,\rU$ are defined as the case $\mathrm{II}$ of Definition \ref{ddef 5}. As in the proof of part $(2)$, the group $\mathrm{Q}$ has two orbits on $\mathrm{P}\backslash\rG$: the closed one $\mathrm{P}\cdot e$ and the open one $\mathrm{P}^{-1}\cdot \omega_0$. The condition $(4)$ can be justified as part $(2)$, and condition $(\ast)$ is clear since $\omega_0(\rU)\cap\mathrm{V}=\mathds{1}$. Now we apply theorem $5.2$ \cite{BeZe}. The functor corresponds to the orbit $\mathrm{P}\cdot e$ is $\ind_{m}^{m-1}\circ\Phi_{\theta,m-1}^-$ by noticing $\varepsilon\circ\Phi_{\theta,m-1}^-\cong\Phi_{\theta,m-1}\circ\varepsilon$. Now we consider the functor corresponds to the orbit $\mathrm{P}\cdot\omega_0^{-1}$. Following the notation as \S $5.1$ \cite{BeZe}, the character $\psi'=\omega_0^{-1}(\psi)\vert_{\rM\cap\omega_0^{-1}(\mathrm{V})}$ is trivial. The character $\varepsilon_1$ is trivial, and $\varepsilon_2\cong\varepsilon^{-1}$. Hence the functor corresponded to this orbit is
$$\ind_{m}^{m-1}\circ\res_{P'}^{\rM_{n_1,\ldots,n_{m-1},n_m}'}\circ\Psi_{m-1}^-,$$
from which we deduce the exact sequence desired.
\end{proof}

\begin{cor}
\label{dcor 6}
The functor $\ind_m^{m-1}$ is defined respectively to the corresponding cases as in Definition \ref{ddef 5}.
\begin{enumerate}
\item Let $\rho\in\mathrm{Rep}_k(P_{(n_1,\ldots,n_m),m}')$, and $\theta,\dot{\theta}$ be as in Proposition \ref{dprop 5} part $(2)$. Assume that $1\leq i\leq n_m$, then we have an equivalence about taking $i$-th derivative
$$(\ind_{m}^{m-1}\rho)_{\dot{\theta}, m-1}^{(i)}\cong\ind_{m}^{m-1}\rho_{\theta,m}^{(i)};$$
\item Let $\rho\in\mathrm{Rep}_k(P_{(n_1,\ldots,n_m),m-1}')$. Assume that $1\leq i\leq n_{m-1}+n_m$, then $(\ind_{m}^{m-1}\rho)_{\dot{\theta},m-1}^{(i)}$ is filtrated by
$\ind_{m}^{m-1}((\rho_{\theta,m}^{(i-j)})_{\theta,m-1}^{(j)})$, where $[1,i-n_m]\leq j\leq i$ and $[1,i-n_m]$ denotes the bigger integer;
\item Let $\rho\in\mathrm{Rep}_k(\rM_{n_1,\ldots,n_m}')$. Assume that $i\geq 0$, then $(\ind_{m}^{m-1}\rho)_{\dot{\theta},m-1}^{(i)}$ is filtrated by $\ind_{m}^{m-1}((\rho_{\theta,m}^{(i-j)})_{\theta,m-1}^{(j)})$, where $[1,i-n_m]\leq j\leq i$;
\item Let $\rho\in\mathrm{Rep}_k(\rM_{n_1,\ldots,n_m}')$, there is an equivalence:
$$(\ind_2^1\circ\cdots\circ\ind_{m-1}^{m-2}\circ\ind_{m}^{m-1}\rho)_{\dot{\theta},1}^{(n_1+\cdots+n_m)}\cong(\cdots((\rho_{\theta,m}^{(n_m)})_{\theta,m-1}^{(n_{m-1})})\cdots)_{\theta,1}^{(n_1)}.$$
\end{enumerate}
\end{cor}

\begin{proof}
Part $(1)$ follows from the exactness of $\Phi_{\theta,m}^-,\Psi_m^-$ and \ref{dprop 5} $(2)$; $(2)$ from $(1)$ and \ref{dprop 5} $(3)$, $(3)$ from $(1),(2)$ and \ref{dprop 5} $(1)$. Part $(4)$ follows from $(3)$, by noticing that
$$\ind_2^1\circ\cdots\circ\ind_{m-1}^{m-2}\circ\ind_{m}^{m-1}\rho\cong i_{\rM_{n_1,\ldots,n_m}}^{\mathrm{GL}_{n_1+\cdots+n_m}}\rho.$$
In fact, this is the transitivity of parabolic induction.
\end{proof}

\subsection{Uniqueness of supercuspidal support}
We take the same notations as in Section \ref{subsubsection 5.1.1}.
\begin{prop}
\label{Uprop 1}
Let $\tau\in\mathrm{Rep}_{k}(\rM_{n_1,\ldots,n_m}')$, and $\theta$ a non-degenerate character of $\rU_{n_1,\ldots,n_m}$. Then $\tau_{\theta,m}^{(n_1+\ldots+n_m)}\neq 0$ is equivalent to say that $\Hom_{k[\rU_{n_1,\ldots,n_m}]}(\tau,\theta)\neq 0$. In particular, this is equivalent to say that $(\rU_{n_1,\ldots,n_m},\theta)$-coinvariants of $\tau$ is non-trivial.
\end{prop}

\begin{proof}
In this proof, we use $\rU$ to denote $\rU_{n_1,\ldots,n_m}$. For the first equivalence, notice that $\Phi_{\theta,m}^-(\tau)\neq 0$ is equivalent to say that $(V_{n_m-1},\theta)$-coinvariants of $\tau$ is non-trivial. For $1\leq s\leq n_m-1$, let $V_{s}$ denote the subgroup of $\rU$ consisting with the matrices with non-zero coefficients only on the $(s+1)$-th line and the diagonal. Let $W$ denote the representation space of $\tau$. The space of $\tau_{\theta,m}^{(n_1+\cdots+n_m)}$ is isomorphic to the quotient of $W$ by the subspace $W_{\theta}$ generated by $g_s(w)-\theta(g_s)w$, for every $s$ and $g_s\in V_{s}$, $w\in W$. Meanwhile, since the subgroups $V_{s}$'s generate $\rU$, and $\theta$ is determined by $\theta\vert_{V_{s}}$ while considering every $s$, the subspace $W_{\theta}$ of $W$ is isomorphic to the subspace generated by $g(w)-\theta(g)w$, where $g\in\rU$. Hence $\tau_{\theta,m}^{(n_1+\cdots+n_m)}\neq 0$ is equivalent to say that $(\rU_{n_1,\ldots,n_m},\theta)$-coinvariants of $\tau$ is non-trivial. The second equivalence is clear, since the $(\rU,\theta)$-coinvariants of $\tau$ is the largest quotient of $\tau$ such that $\rU$ acts as a multiple of $\theta$. 
\end{proof}

\begin{prop}
\label{Uprop 2}
Let $\tau\in\mathrm{Rep}_{k}(\rM_{n_1,\ldots,n_m}')$, and $\rho$ be a subquotient of $\tau$. Let $\theta$ be a non-degenerate character of $\rU_{n_1,\ldots,n_m}$, and $\rho_{\theta,m}^{(n_1+\cdots+n_m)}$ is non-trivial, then $\tau_{\theta,m}^{(n_1+\cdots+n_m)}$ is non-trivial.
\end{prop}

\begin{proof}
We consider the $(n_1+\cdots+n_m)$-th derivative functor corresponding to the non-degenerate character $\theta$, from the category $\mathrm{Rep}_{k}(\rM_{n_1,\cdots,n_m}')$ to the category of $k$-vector spaces, which maps $\tau$ to $\tau_{\theta,m}^{(n_1+\cdots+n_m)}$. By Definition \ref{ddef 4} and Remark \ref{drem 3}, this functor is a composition of functors $\Psi_{\cdot}^{-}$ and $\Phi_{\theta,\cdot}^{-}$, hence is exact, from which we deduce that $\tau_{\theta,m}^{(n_1+\cdots+n_m)}$ is non-trivial.
\end{proof}

\begin{thm}
\label{Uthm 3}
Let $\rM'$ be a Levi subgroup of $\rG'$, and $\rho$ an irreducible $k$-representation of $\rM'$. The supercuspidal support of $\rho$ is a $\rM'$-conjugacy class of one unique supercuspidal pair.
\end{thm}

\begin{proof}
Since the cuspidal support of irreducible $k$-representation is unique, to prove the uniqueness of supercuspidal support, it is sufficient to assume that $\rho$ is cuspidal. Let $\pi$ be an irreducible cuspidal $k$-representation of $\rM$, such that $\rho$ is a sub-representation of $\res_{\rM'}^{\rM}\pi$. Let $(\rL,\tau)$ be a supercuspidal pair of $\rM$, and $[\rL,\tau]$ consists the supercuspidal support of $\pi$.  We have $\res_{\rL'}^{\rL}\tau\cong\oplus_{i\in I}\tau_i$, where $I$ is a finite index set. According to Proposition \ref{Lproposition 10}, the supercuspidal support of $(\rM',\rho)$ is contained in the union with respect to $i\in I$ of $\rM'$-conjugacy class of $(\rL',\tau_i)$. To finish the proof of our theorem, it remains to prove that there exists one unique $i_0\in I$ such that $(\rL',\tau_{i_0})$ is contained in the supercuspidal support of $(\rM',\rho)$.

After conjugation by $\rG'$, we could assume that $\rM'=\rM_{n_1,\ldots,n_m}'$ and $\rL'=\rM'_{k_1,\ldots,k_l}$ for a familly of integers $m,l, n_1,\ldots,n_m,k_1,\ldots,k_l\in\mathbb{N}^{\ast}$. There exists a non-degenerate character $\theta$ of $\rU=\rU_{n_1,\ldots,n_m}$, such that $\rho_{\theta,m}^{(n_1+\cdots+n_m)}\neq 0$. In fact, let $\theta$ be any non-degenerate character of $\rU$ and write $\res_{\rM'}^{\rM}\pi\cong\oplus_{s\in S}\pi_s$, where $S$ is a finite index set. We have:
$$\pi^{(n_1+...+n_m)}\cong(\pi\vert_{\rM'})_{\theta,m}^{(n_1+\cdots+n_m)}\cong\oplus_{s\in S}(\pi_s)_{\theta,m}^{(n_1+\cdots+n_m)},$$
where $\pi^{(n_1+\cdots+n_m)}$ indicates the $(n_1+...+n_m)$-th derivative of $\pi$. As in Section [\S III, 5.10, 3)] of \cite{V1}, we have $\mathrm{dim}(\rho^{(n_1+\cdots+n_m)})=1$, hence there exists one element $s_0\in S$ such that $(\pi_{s_0})_{\theta,m}^{(n_1+\cdots+n_m)}\neq 0$. Notice that $\rho$ are isomorphic to some $\pi_s$, which means there exists a diagonal element $t\in\rM$, such that the $t$-conjugation $t(\pi_{s_0})\cong\rho$. The character $t(\theta)$ is also non-degenerate of $\rU$, and we have $(t(\pi_{s_0}))_{t(\theta),m}^{(n_1+\cdots+n_m)}\cong (\pi_{s_0})_{\theta,m}^{(n_1+\cdots+n_m)}$ as $k$-vector spaces. We conclude that $\mathrm{dim}\rho_{t(\theta),m}^{(n_1+\cdots+n_m)}=1$. To simplify the notations, we assume $t=1$.

If $\rho$ is a subquotient of $i_{\rL'}^{\rM'}\tau_i$ for some $i\in I$. By (4) of Corollary \ref{dcor 6} and Proposition \ref{Uprop 2}, the derivative $\tau_{i_{\theta,m}}^{(n_1+\cdots+n_m)}\neq 0$. By section [\S III, 5.10, 3)] of \cite{V1}, the derivative $\tau_{\theta,m}^{(n_1+\cdots+n_l)}=1$, which means the dimension of $(\rU_{\rL'},\theta)$-coinvariants of $\tau$ is $1$ (by \ref{Uprop 1}). Notice that the $(\rU_{\rL'},\theta)$-coinvariants of $\tau$ is the direct sum of $(\rU_{\rL'},\theta)$-coinvariants of $\tau_i$ for every $i\in I$, which implies that there exists one unique $i_0\in I$ whose $(\rU_{\rL'},\theta)$-coinvariants is non-zero with dimension $1$. By Proposition \ref{Uprop 1} and Proposition \ref{Uprop 2}, this is equivalent to say that there exists one unique $i_0\in I$, such that the derivative $\tau_{i_{\theta,m}}^{(n_1+\cdots+n_m)}\neq 0$.
\end{proof}

\pagestyle{empty}


\end{document}